\documentclass[a4paper,12pt]{article}
\usepackage[top=1.25in,left=1in,right=1in,bottom=1.25in]{geometry}           
\usepackage{latexsym}
\usepackage{epsfig, ecltree, epic, eepic}
\usepackage{enumerate,amsmath,amssymb,dsfont,pifont,amsthm}
\usepackage[dvips]{color}
\usepackage{graphicx}
\usepackage[arrow, matrix, curve]{xy}

\theoremstyle{plain}
\newtheorem{theorem}{Theorem}[section]
\newtheorem{corollary}[theorem]{Corollary}

\newtheorem{proposition}[theorem]{Proposition}

\theoremstyle{definition}

\theoremstyle{remark}
\newtheorem{remark}[theorem]{Remark}
\newtheorem{remarks}[theorem]{Remarks}
\newtheorem{example}[theorem]{Example}

\newcommand{\bbc}{\mathbb{C}}
\newcommand{\bbr}{\mathbb{R}}

\newcommand{\bbp}{\mathbb{P}}
\newcommand{\bbe}{\mathbb{E}}

\newcommand{\bbn}{\mathbb{N}}

\newcommand{\cf}{\mathcal{F}}

\newcommand{\abs}[1]{\left| #1 \right|}
\newcommand{\norm}[1]{\left\| #1 \right\|}

\newcommand{\roi}[2]{\left[ #1 , #2 \right[}

\newcommand{\ci}[2]{\left[ #1, #2 \right]}

\newcommand{\rosi}[2]{\roi{[#1}{#2[}}

\newcommand{\csi}[2]{\ci{[#1}{#2]}}

\begin{document}

\allowdisplaybreaks

\title{\bfseries SDEs Driven By SDE Solutions}

\author{%
    \textsc{Alexander Schnurr$^*$}\\[5mm]
    \small Lehrstuhl IV, Fakult\"at f\"ur Mathematik, Technische Universit\"at Dortmund,\\[-2mm]
\small D-44227 Dortmund, Germany, alexander.schnurr@math.tu-dortmund.de 
}

\date{\today}

\maketitle
\begin{abstract}
We consider stochastic differential equations (SDEs) driven by Feller processes which are themselves solutions of multivariate L\'evy driven SDEs. The solutions of these `iterated SDEs' are shown to be non-Markovian. However, the process consisting of the driving process and the solution is Markov and even Feller in the case of bounded coefficients. The generator as well as the semimartingale characteristics of this process are calculated explicitly and fine properties of the solution are derived via the probabilistic symbol. A short simulation study and an outlook in the direction of stochastic modeling round out the paper. 
\end{abstract}



\section{Introduction and General Setting}

It is a well known fact that, under appropriate regularity conditions (e.g. Lipschitz) of the  coefficient $\Phi$, the multivariate L\'evy driven stochastic differential equation (SDE)
\begin{align} \begin{split} \label{levysde}
  dY_t&= \Psi(Y_{t-}) \, dZ_t \\
   Y_0&= y
\end{split} \end{align}
admits a unique strong solution on a suitable extension of the original stochastic basis, i.e. the one on which the driving $m$-dimensional L\'evy process is defined (cf. \cite{protter} Section V.6). In \cite{sdesymbol} it was shown that if $\Psi:\bbr^n\to\bbr^{n\times m}$ is Lipschitz continuous and bounded the solution $Y$ is a Feller process with symbol
\[
  p(y,\xi)=\psi(\Psi(y)'\xi)
\]
where $\psi:\bbr^m\to\bbc$ is the symbol, i.e. the characteristic exponent, of the driving process. 
In particular the process $Y$ as well as the symbol $p$ are state-space dependent which allows for more flexibility in stochastic modeling. In \cite{sdesymbol} the symbol and the related indices where used to obtain fine properties of the solution $Y=(Y_t)_{t\geq 0}$. It is a natural question to ask whether this procedure can be iterated, i.e. use the process $Y$ itself as a driving term and consider the SDE
\begin{align} \begin{split} \label{itosde}
  dX_t&= \Phi(X_{t-}) \, dY_t \\
   X_0&= x
\end{split} \end{align} 
where $\Phi:\bbr^d\to\bbr^{d\times n}$ is again Lipschitz continuous. In Section 2 we show that the process $X$ is in general not Markovian. This means in particular that it does not have a symbol - at least not in the classical sense -, since it does not have a generator. However, we will show in Section 3 how the symbol of the bivariate process $(X,Y)$ can be used to obtain fine properties of $X$ in the bounded-coefficient case. To our knowledge this is the first time that the symbol is used to analyze fine properties of a non-Markovian process. Section 4 contains a simulation study which we have included to illustrate what the process $Z$ might look like. The last section contains an outlook toward stochastic modeling of financial data. In particular we set our class of processes in relation to other models which where introduced recently.

Since some of the process classes are not defined in a unique way, let us first fix some terminology: a Markov process $(\Omega, \cf, (\cf_t)_{t\geq 0}, (X_t)_{t\geq 0} ,\bbp^x)_{x\in\bbr^d}$ is defined in the sense of Blumenthal-Getoor (cf. \cite{blumenthalget}) this means in particular that the following formula holds
\begin{align} \label{universalmp}
  P_{s,t}^w(x,A) = P_{0,t-s}^{x}(x,A)=:P_{t-s}(x,A)
\end{align}
where $P_{s,t}^w(x,A)$ is the regular version of $\bbp^w(X_t\in A \, | \, X_s=x)$ with $w,x\in\bbr^d$, $s\leq t$ and $A$ is a Borel set in $\bbr^d$. In order to emphasize the difference to Markov processes which are only defined for a single starting point, we call these processes universal Markov (cf. \cite{bauer}, \cite{jacob3}).
As usual, we associate a semigroup $(T_t)_{t\geq 0}$ of operators on $B_b(\bbr^d)$ with every such process by setting
\[
    T_t u(x):= \bbe^x u(X_t), \quad t\geq 0,\; x\in \bbr^d.
\]
Denote by $C_\infty(\bbr^d)$ the space of all
functions $u:\bbr^d\to\bbr$ which are continuous and vanish at
infinity, i.e. $\lim_{\norm{x}\to\infty}u(x) =0$; then
$(C_\infty(\bbr^d),\norm{\cdot}_\infty)$ is a Banach space and $T_t$ is for
every $t$ a contractive, positivity preserving and sub-Markovian
operator on $B_b(\bbr^d)$. We call $(T_t)_{t\geq 0}$ a Feller
semigroup and $(X_t)_{t\geq 0}$ a Feller process, if the semigroup is strongly continuous and if the following condition is satisfied:
\begin{align}
  T_t:C_\infty(\bbr^d) \to C_\infty(\bbr^d) \text{ for every }t\geq 0.
\end{align}

Without loss of generality we assume all Feller processes we encounter to be c\`adl\`ag (cf. \cite{revuzyor} Theorem III.2.7). The generator of the Feller process $(A,D(A))$ is the closed operator given by
\begin{gather}\label{generator}
    Au:=\lim_{t \downarrow 0} \frac{T_t u -u}{t} \qquad\text{for\ \ } u\in D(A)
\end{gather}
where the domain $D(A)$ consists of all $u\in C_\infty(\bbr^d)$ for which the limit \eqref{generator} exists uniformly. A Feller process is called rich if $C_c^\infty(\bbr^d)\subseteq D(A)$.
By a classical result due to P.\ Courr\`ege \cite{courrege} the generator of a rich Feller process is (restricted to the test functions $C_c^\infty(\bbr^d)$) a pseudo differential operator with symbol $-p(x,\xi)$, i.e.\ $A$ can be written as
\begin{gather} \label{pseudo}
    Au(x)= - \int_{\bbr^d} e^{ix'\xi} p(x,\xi) \widehat{u}(\xi) \,d\xi, \qquad u\in C_c^\infty(\bbr^d)
\end{gather}
where $\widehat{u}(\xi)=(2\pi)^{-d}\int e^{-iy'\xi}u(y) dy$ denotes the Fourier transform. The function $p:\bbr^d \times \bbr^d \to \bbc$ is locally bounded and, for fixed $x$, a continuous negative definite function in the sense of Schoenberg in the co-variable $\xi$ (cf. \cite{bergforst} Chapter II). This means it admits a L\'evy-Khintchine representation
\begin{align} \label{lkfx}
  p(x,\xi)=
  -i \ell'(x)  \xi + \frac{1}{2} \xi'Q(x) \xi -\int_{w\neq 0} \left( e^{i \xi' w} 
  -1 - i \xi' w \cdot \chi(w)\right)N(x,dw)
\end{align}
where $\ell(x)=(\ell^{(j)}(x))_{1\leq j \leq d} \in \bbr^d$, $Q(x)=(Q^{jk}(x))_{1\leq j,k \leq d}$ is a symmetric positive semidefinite matrix and $N(x,dw)$ is a measure on $\bbr^d\setminus\{0\}$ such that $\int_{w\neq 0} (1 \wedge \norm{w}^2) \,N(x,dw) < \infty$ and $\chi$ is a cut-off function (see below). The function $p:\bbr^d\times \bbr^d\to \bbc$ which is often written as $p(x,\xi)$ is called the symbol of the operator. For details on the rich theory of the interplay between processes and their symbols we refer the reader to \cite{jacob1,jacob2,jacob3}.

A universal Markov process $X$ is called Markov semimartingale, if $X$ is for every $\bbp^x$ a semimartingale. In \cite{mydiss} it is shown that every rich Feller process is a Markov semimartingale and even an It\^o process in the sense of \cite{vierleute} and that the triplet $(\ell(x),Q(x),N(x,dw))$ appears again in the semimartingale characteristics
\begin{align*}
  B^X_t(\omega)&=\int_0^t \ell(X_s(\omega)) \ ds \\
  C^X_t(\omega)&=\int_0^t Q (X_s(\omega)) \ ds \\
  \nu^X(\omega;ds,dw)&=N(X_s(\omega),dw)\ ds.
\end{align*}
It\^o processes admit an extended generator (cf. \cite{vierleute} Definition 7.1) which coincides with the ordinary generator on $C_c^2(\bbr^d)$ for rich Feller processes. The probabilistic symbol of $X$ is defined as follows (cf. \cite{mydiss} Definition 4.3): fix a starting point $x$ and define $\sigma=\sigma^x_R$ to be the first exit time from the ball of radius $R > 0$ with respect to $\bbp^x$ $(x\in\bbr^d)$:
\begin{align} \label{stoppingtime}
  \sigma:=\sigma^x_R:=\inf\big\{t\geq 0 : \norm{X_t-x} > R \big\}.
\end{align}
The function $p:\bbr^d\times \bbr^d \rightarrow \bbc$ given by
  \begin{gather} \label{symbol}
    p(x,\xi):=- \lim_{t\downarrow 0}\bbe^x \left(\frac{e^{i(X^\sigma_t-x)'\xi}-1}{t}\right)
  \end{gather}
is called the probabilistic symbol of the process, if the limit exists for every $x,\xi\in\bbr^d$ independently of the choice of $R>0$. In \cite{mydiss} Theorem 4.4 it was shown that the probabilistic symbol exists for every It\^o process for which the functions $x\mapsto\ell(x)$, $x\mapsto Q(x)$ and $x\mapsto \int_{w\neq 0} (1\wedge \norm{w}^2) \, N(x,dw)$ are finely continuous (cf. \cite{blumenthalget} Section II.4) and locally bounded. Let us emphasize that these assumptions are very weak, even weaker than ordinary continuity which is in turn fulfilled by virtually all examples of It\^o processes in the literature. Furthermore it was shown that the probabilistic symbol coincides with the ordinary symbol in the rich-Feller-case. A posteriori this justifies the name.

An important subclass of rich Feller processes is the class of L\'evy  processes. These are c\`adl\`ag processes which have stationary and independent increments. The symbol of a L\'evy process $Z$ does not depend on the starting point and coincides with the characteristic exponent $\psi$ (cf. \cite{kyprianou} Definition 1.5) which is given by
\[
  \bbe^0 e^{i Z_t'\xi}=e^{-t\psi(\xi)}.
\]
It is a well known fact that there is a 1:1-correspondence between L\'evy processes and continuous negative definite functions (cf. e.g. \cite{sato}). In the case of nice Feller processes the problem whether a state-space dependent symbol yields a process is part of ongoing research (cf. \cite{hoh00}, \cite{hoh02} and for a survey \cite{jac-sch-survey}).

Most of the notation we are using is standard. In the context of semimartingales we mainly follow \cite{jacodshir}. The only difference is that we write $\chi(y)\cdot y$ for the truncation function, where $\chi$ is measurable with compact support and equal to 1 in a neighborhood of zero. A possible way to choose the cut-off functions $\chi$ in different dimensions $m\in\bbn$ is as follows: take a one-dimensional cut-off function $\chi:\bbr\to\bbr$ and define for $x\in\bbr^m$: $\tilde{\chi}(x):=\chi(x^{(1)})\cdots \chi(x^{(m)})$. Vectors are column vectors. The transposed vector or matrix is written as $v'$ or $Q'$. In Section 2 we will consider vector-valued processes consisting of $X_t\in\bbr^d$ and $Y_t\in\bbr^n$. In this context we write $(X,Y)$ instead of $(X',Y')'$ in order to keep the notation simple. 

\section{Markov Solutions of SDEs}

As described in the introduction we want to use the symbol in order to analyze the process $X$ given as the solution of equation \eqref{itosde} which we restate here in a slightly different form in order to emphasize the dependence on the starting points
\begin{align*} \tag*{(2)*} \begin{split} 
  dX^{x,y}_t&= \Phi(X^{x,y}_{t-}) \, dY^y_t \\
   X^{x,y}_0&= x.
\end{split} \end{align*} 
In this section we will be more general: instead of restricting us to solutions of \eqref{levysde}  the driving term $Y$ is defined to be an It\^o process in the sense of \cite{vierleute}.  This encompases the setting described in Section 1 by Theorem 3.10 of \cite{mydiss}, which states that every rich Feller process is an It\^o process. Let us remark that this more general setting is still in the spirit of the title of the present paper since every It\^o process is a solution of an SDE of Skorokhod type (cf. \cite{cinlarjacod} Theorem 3.13) on a suitable extension of the underlying probability space. Since we want to use concepts like the generator and the symbol of the process, we have to deal with Markov semimartingales. The problem we are facing is due to Theorem 1 of \cite{jacodprotter}:

\begin{theorem} \label{thm:mustbelevy}
Let $Y$ be a Markov semimartingale and $\Phi:\bbr\to\bbr$ Borel measurable and never zero. If the equation \eqref{itosde} admits a unique solution $X$ which is a Markov process such that \eqref{universalmp} holds then $Y$ is a L\'evy process.
\end{theorem}


This means that if we want to have a solution which is in the class of processes we are interested in, we just obtain an indexed family $(X^{x,y})_{y\in\bbr^n}$ of which every member is not more general than a solution of \eqref{levysde}. In particular the case described in Section 1 would be excluded, since $Y$ is a L\'evy process only in the trivial case where $\Psi$ is constant. In order to get more general Markov semimartingales one has to study the ($d+n$)-dimensional process $(X,Y)$ instead. This approach is motivated by the following result:

\begin{theorem}
Let $Y$ be an It\^o process in the sense of \cite{vierleute} such that the functions $x\mapsto\ell(x)$, $x\mapsto Q(x)$ and $x\mapsto \int_{w\neq 0} (1\wedge \norm{w}^2) \, N(x,dw)$ are finely continuous and locally bounded. Let $\Phi:\bbr^d\to\bbr^{d\times n}$ be bounded and locally Lipschitz. Denote the probabilistic symbol of $Y$ by $p(y,\eta)$. In the above setting the $(d+n)$-dimensional process $V:=(X,Y)$ consisting of the solution and the driving term of \eqref{itosde} is an It\^o process in the sense of \cite{vierleute} with probabilistic symbol $q:\bbr^{d+n}\times \bbr^{d+n} \to \bbc$ given by
\begin{align} \label{bivsymbol}
  q\left( \binom{x}{y} , \binom{\xi_1}{\xi_2} \right) = p\big( y, \Phi(x)'\xi_1+\xi_2 \big).
\end{align}
\end{theorem}

\begin{proof}
By Theorem (8.11) of \cite{vierleute} in the version of Theorem 2.52 of \cite{mydiss} the process $(X,Y)$ is a Markov semimartingale. 
Let us denote the differential characteristics of $Y$ with respect to $\chi^Y:\bbr^n\to \bbr$ by $(\ell,Q,N(\cdot,dy))$. To obtain the structure of the characteristics we use \cite{jacodshir} Proposition IX.5.3, written in a suitable form. It shows that the characteristics $(B^V,C^V, \nu^V)$ of the process $V=(X,Y)$ with respect to the cut-off function $\chi^V:\bbr^{d+n}\to\bbr$ are
\begin{align*}
B^V_t&=\begin{pmatrix} \int_0^t \Phi(X_s) \ell(Y_s) \, ds \\ \int_0^t \ell(Y_s) \, ds\end{pmatrix} \\[12mm]
   &+ \int_0^t \int_{\bbr^n \backslash \{0\}} \begin{pmatrix} \Phi(X_{s-})w \\ w \end{pmatrix} \Bigg(\chi^V \begin{pmatrix} \Phi(X_{s-})w \\ w \end{pmatrix} - \chi^Y(w)  \Bigg)N(Y_s,dw) \, ds \\[8mm]
C_t^V&=\left(\begin{array}{c|c}
   \rule[-5mm]{0mm}{10mm}\int_0^t \Phi(X_s) Q(Y_s)(\Phi(X_s))' \, ds  
  &\int_0^t \Phi(X_s)Q(Y_s) \, ds \\ 
     \hline 
   \rule[0mm]{0mm}{8mm}\int_0^t Q(Y_s)(\Phi(X_s))' \, ds 
  &\int_0^t Q(Y_s) \, ds \end{array}\right) \in
  \left(\begin{array}{c|c}
                   \bbr^{d\times d}  &\bbr^{d\times n}  \\ 
     \hline        \bbr^{n\times d}  &\bbr^{n\times n}   \end{array}\right) \\[12mm]
&\hspace*{-7mm}\nu^V(\omega,ds,dw)= \displaystyle f_*^{\omega,s}  (N(Y_s(\omega),dw)) \, ds 
\end{align*}
where
\[
f_*^{\omega,s}(N(Y_s(\omega),dw)) = N(Y_s(\omega), f^{\omega,s} \in dw)
\]
and 
\[
f^{\omega,s}(y):= \binom{\Phi(X_{s-}(\omega))y}{y}.
\]
Switching between the left continuous and the right continuous version of the processes $X$ and $Y$ is possible since we are integrating with respect to Lebesgue measure and since the number of jumps of the processes is countable. The structure of the semimartingale characteristics shows that, the $(d+n)$-dimensional process $V$ is again It\^o.

As the process $V$ is an It\^o process, it admits a symbol. We calculate the symbol in the case $d=n=1$, since the multidimensional version is proved similarly, but the notation becomes more involved. Let $\sigma$ be the stopping time defined in \eqref{stoppingtime}. We apply It\^o's formula for semimartingales (cf. \cite{protter} Theorem II.33) on the function $\exp(i(\cdot-v)'\xi)$ where $v$ denotes the starting point $(x,y)$ of the bivariate process $V:=(X,Y)$. In order to keep the notation simple we write
\[
E(\xi,s):= e^{i(V_{s}^\sigma - v)'\xi} = e^{i (X_s^{\sigma}-x)\xi_1 + i(Y_s^{\sigma}-y)\xi_2} 
\]
and obtain
\begin{equation} \begin{aligned} \label{maineq} 
\frac{1}{t} \bbe^v ( E(\xi,t) -1) &= \underbrace{\frac{1}{t} \bbe^v \int_0^t i\xi_1 E(\xi,s-) \, dX_s^\sigma}_{I}
        + \underbrace{\frac{1}{t} \bbe^v \int_0^t i\xi_2 E(\xi,s-) \, dY_s^\sigma}_{II} \\
       &\hspace{-20mm}-\underbrace{\frac{1}{t} \bbe^v\frac{1}{2} \int_0^t \xi_1^2 E(\xi,s-) \, d[X^\sigma,X^\sigma]^c_s}_{III}
       -             \underbrace{\frac{1}{t} \bbe^v\int_0^t \xi_1\xi_2 E(\xi,s-) \, d[X^\sigma,Y^\sigma]^c_s}_{IV}\\
       &\hspace{-20mm}-\underbrace{\frac{1}{t} \bbe^v\frac{1}{2} \int_0^t \xi_2^2 E(\xi,s-) \, d[Y^\sigma,Y^\sigma]^c_s}_{V} \\
       &\hspace{-20mm}+\underbrace{\frac{1}{t} \bbe^v\sum_{0\leq s \leq t} \Big( E(\xi,s-)(e^{i (\Delta X_s^{\sigma})\xi_1 + i(\Delta Y_s^{\sigma})\xi_2}-1-i\xi_1 \Delta X^\sigma_s - i\xi_2 \Delta Y^\sigma_s) \Big)}_{VI}.
\end{aligned} \end{equation}
By the canonical representation of semimartingales (cf. \cite{jacodshir} II.2.35) we can write $Y$ in the following way 
\[
Y=y+Y_t^{c}+(\chi\cdot \text{id})*(\mu^Y-\nu^Y)_t + (\text{id}-\chi\cdot \text{id}) * \mu_t^{Y}+\int_0^t \ell(Y_s) \, ds.
\]
where (id:\ $w\mapsto w$). For the remainder of the proof we fix the cut-off function $\chi(w)=1_{\{\abs{w}\leq 1\}}$ in order to simplify the calculation. Therefore, we obtain for the first term of \eqref{maineq}
\begin{align}
  \frac{1}{t}\,\bbe^v &\int_{0}^{t} \Big(i \xi_1 E(\xi,s-)\Big) \,dX^\sigma_s \notag\\
  &= \frac{1}{t}\,\bbe^v \int_{0}^{t} \Big(i \xi_1 E(\xi,s-)\Big) d\left(\int_0^s\Phi(X_{r-}) 1_{\csi{0}{\sigma}}(\cdot,r) \,dY_r\right)\notag\\
  &= \frac{1}{t}\, \bbe^v \int_{0}^{t} \Big(i \xi_1 E(\xi,s-)
        \Phi(X_{s-}) 1_{\csi{0}{\sigma}}(\cdot,s) \Big)\,dY_s\notag\\
  &= \frac{1}{t}\, \bbe^v \int_{0}^{t} \Big(i \xi_1 E(\xi,s-)
        \Phi(X_{s-}) 1_{\csi{0}{\sigma}}(\cdot,s) \ell(Y_{s-})\Big) \,ds\label{drift}\\
  &\qquad  + \frac{1}{t}\, \bbe^v \int_{0}^{t} \Big(i \xi_1 E(\xi,s-)
        1_{\csi{0}{\sigma}}(\cdot,s)\Big) d\left(\sum_{0\leq r\leq s} (\Phi(X_{r-})\Delta Y_r 1_{\{\abs{\Delta Y_r}> 1 \}}\right)\label{jumps}
\end{align}
since the integrals with respect to the two other terms are martingales by the following considerations:
\begin{align*}
\left[\int_0^\cdot E(\xi,s-) \, dX_s^{\sigma,c}, \int_0^\cdot E(\xi,s-) \, dX_s^{\sigma,c} \right]_t 
&= \left[\int_0^\cdot E(\xi,s-) \, dX_s^{c}, \int_0^\cdot E(\xi,s-) \, dX_s^{c} \right]_t^\sigma \\
&\hspace{-25mm}= \int_0^t (E(\xi,s-))^2 1_{\csi{0}{\sigma}}(\cdot,s)  (\Phi(X_{s-}))^2 \, d[Y^{c},Y^{c}]_s \\ 
&\hspace{-25mm}= \int_0^t (E(\xi,s-))^2 1_{\csi{0}{\sigma}}(\cdot,s) (\Phi(X_{s-}))^2\, d\left(\int_0^s Q(Y_{r}) \ dr \right) \\ 
&\hspace{-25mm}= \int_0^t (E(\xi,s-))^2 1_{\rosi{0}{\sigma}}(\cdot,s)Q(Y_s)  (\Phi(X_{s-}))^2\, ds
\end{align*}
where we used several well known facts about the square bracket. The last term is uniformly bounded in $\omega$ and therefore
\[
\left[\int_0^\cdot E(\xi,s-) \, dX_s^{\sigma,c}, \int_0^\cdot E(\xi,s-) \, dX_s^{\sigma,c} \right]_t  < \infty \hspace{10mm}\text{ for every } t\geq 0.
\] 
It follows from Corollary 3 of Theorem II.27 of \cite{protter} that $\int_0^t E(\xi,s-) \, dX_s^{\sigma,c}$ is a martingale which is zero at zero and therefore, its expected value is constantly zero. The same is true for the integral with respect to the compensated sum of small jumps, since the function
\[
E(\xi,s-) \ y  \chi(y) \  \Phi(X_{s-}) 1_{\csi{0}{\sigma}}(\cdot,s)
\]
is in the class $F_p^2$ of Ikeda and Watanabe (cf. Section II.3 of \cite{ikedawat}).

The second term of \eqref{maineq} works alike. Putting the continuous parts of terms I and II together we obtain
\begin{align*}
\frac{1}{t}\, \bbe^v &\int_{0}^{t} \left(i \binom{\xi_1}{\xi_2}'\binom{\Phi(X_{s-})}{1} E(\xi,s-)
         1_{\csi{0}{\sigma}}(\cdot,s) \ell(Y_{s-})\right) \,ds \\
&=\frac{1}{t}\, \bbe^v \int_{0}^{t} \left(i \binom{\xi_1}{\xi_2}'\binom{\Phi(X_{s})}{1} E(\xi,s)
         1_{\rosi{0}{\sigma}}(\cdot,s) \ell(Y_{s})\right) \,ds \\
&\xrightarrow[t\downarrow 0]{}  i \binom{\xi_1}{\xi_2}'\binom{\Phi(x)}{1} \ell(y)
\end{align*}
where we have used that the countable number of jump times form a set of Lebesgue measure zero, that the processes $X$ and $Y$ are bounded on $\rosi{0}{\sigma}$ and that $\ell$ is finely continuous (cf. \cite{blumenthalget} Theorem II.4.8).

For term III we have by the same reasoning
\begin{align*}
\bbe^v\frac{-1}{2t} &\int_0^t \xi_1^2 E(\xi,s-) \, d[X^\sigma,X^\sigma]^c_s \\
&=\bbe^v\frac{-1}{2t} \int_0^t \xi_1^2 E(\xi,s-) \, d\Big( \int_0^s (\Phi(X_{r-}))^2 1_{\csi{0}{\sigma}}(\cdot,r)Q(Y_{r}) \, dr \Big) \\
&= -\frac{1}{2} \xi_1^2 \bbe^v \frac{1}{t} \int_0^t  E(\xi,s) \, (\Phi(X_{s}))^2 1_{\rosi{0}{\sigma}}(\cdot,s)Q(Y_{s}) \, ds \\ 
&\xrightarrow[t\downarrow 0]{} -\frac{1}{2} \xi_1^2 (\Phi(x))^2 Q(y)
\end{align*}
and analogously for the terms IV and V
\[
\xi_1\xi_2 \Phi(x) Q(y) -\frac{1}{2} \xi_2^2 Q(y).
\]
Adding the jump parts or terms I and II to term VI we obtain using the short hand
\[
H(\xi,s,w):=E(\xi,s)\left(e^{i (\Phi(X_{s})w)\xi_1 + i w\xi_2}-1-i \binom{\xi_1}{\xi_2}' \binom{\Phi(X_{s})w}{w} 1_{\{\abs{w}\leq 1\}}\right) \, 1_{\rosi{0}{\sigma}}(\cdot,s)
\]
the following (cf. \eqref{jumps})
\begin{align*}
\frac{1}{t}\, \bbe^v &\sum_{0 \leq s \leq t}  E(\xi,s-)\left(e^{i (\Phi(X_{s-})\Delta Y_s^{\sigma})\xi_1 + i(\Delta Y_s^{\sigma})\xi_2}-1-i \binom{\xi_1}{\xi_2}' \binom{\Phi(X_{s-})\Delta Y_s^\sigma}{\Delta Y^\sigma} 1_{\{\abs{\Delta Y_s^\sigma}\leq 1\}}\right)\\
&=\frac{1}{t}\, \bbe^v \int_{[0,t]\times \bbr \backslash \{0 \}} H(\xi,s-,w) \, \mu^Y(\cdot;ds,dw) \\
&=\frac{1}{t}\, \bbe^v \int_{[0,t]\times \bbr \backslash \{0 \}} H(\xi,s-,w) \,  \nu^Y(\cdot;ds,dw) \\
&=\frac{1}{t}\, \bbe^v \int_0^t\int_{ \bbr \backslash \{0 \}} H(\xi,s,w) \, N(Y_s,dw) ds \\
&=\frac{1}{t} \int_0^t \int_{ \bbr \backslash \{0 \}} H(\xi,s,w) \Big(N(Y_s,dw) - N(y,dw)\Big) + \frac{1}{t} \int_0^t \int_{ \bbr \backslash \{0 \}} H(\xi,s,w) N(y,dw) \\
&\xrightarrow[t\downarrow 0]{} 0+ \int_{ \bbr \backslash \{0 \}} \left(e^{i (\Phi(x)w)\xi_1 + i w\xi_2}-1-i \binom{\xi_1}{\xi_2}' \binom{\Phi(x)w}{w} 1_{\{\abs{w}\leq 1\}}\right)
\end{align*}
where we have used that it is possible to integrate `under the expectation' with respect to the compensator of a random measure instead of the measure itself (cf. \cite{ikedawat} Section II.3), if the integrand is in $F_p^1$. This as well as the convergence of the integrals is governed by the well known estimate
\[
\abs{e^{i\xi'w}-1-i\xi'w \chi(w)}\leq const.(1+\norm{\xi}^2)(1\wedge \norm{w}^2)
\]
and the fact that $X$ is bounded on $\rosi{0}{\sigma}$. Hence the result.
\end{proof}

In this case we allow the driving process to be state-space dependent and obtain in the first $d$ components a non-Markov process which is more general than the solution of a L\'evy driven SDE, on which one is thrown back by Theorem \ref{thm:mustbelevy} if one takes only the solution into account. One could speak of $X$ as a hidden Markov model in continuous time with continuous state space.

\begin{remarks} (a) Let us emphasize that the symbol is a canonical object in the following sense: unlike the first characteristic of $(X,Y)$ the symbol does not depend on the choice of the cut-off function.

(b) The $d$-dimensional process $X$ (for fixed $x\in\bbr^d$) is not a homogeneous diffusion with jumps (cf. \cite{jacodshir}). However, the characteristics do admit `differential characteristics' in the sense of \cite{kallsenves}. This is useful if one considers transformations of the process or if one is interested in limit theorems.

(c) Instead of the calculation above one could have also used Theorem 4.4. of \cite{mydiss} which gives a connection between the characteristics of an It\^o process and its symbol. In order to make the present paper more self contained and since this theorem uses some deep results of \cite{vierleute} we decided to include the direct calculation. 
\end{remarks}

Knowing the symbol we can write down the extended generator (cf. \cite{vierleute}) of $V=(X,Y)$  at once by \cite{mydiss} Section 6.1: 
\begin{align*}
\widetilde{A}u(v)=&- \int_{\bbr^{(d+n)}} e^{i v'\xi} q\left( v , \xi \right) \widehat{u}(\xi) \, d\xi = - \int_{\bbr^{(d+n)}} e^{i (x'\xi_1+y'\xi_2)} q\left( \binom{x}{y} , \binom{\xi_1}{\xi_2} \right) \widehat{u}(\xi) \, d\xi\\
     =&\int_{\bbr^{(d+n)}} \bigg(i \ell(y)'(\Phi(x)'\xi_1+\xi_2) - \frac{1}{2} Q(y)(\Phi(x)'\xi_1+\xi_2)(Q(y))' \\
      &+\int_{w\neq 0} \big( e^{i\xi'w}-1-i\xi'w\chi(w) \big) \, f_*^x(N(y,dw)) \bigg)\, d\xi \\
     =&\binom{\Phi(x)'\ell(y)}{\ell(y)} ' \nabla u\left( \binom{x}{y} \right) \\
         &+  \nabla u\left( \binom{x}{y} \right)' \left(\begin{array}{cc} \Phi(x) Q(y)(\Phi(x))'
          & \Phi(x)Q(y)\\ Q(y)(\Phi(x))' & Q(y) \end{array}\right)
            \nabla u\left( \binom{x}{y} \right) \\
      &+\int_{w\neq 0} \Bigg( u\left(\binom{x}{y}+\binom{w_1}{w_2}\right) - u\left(\binom{x}{y} \right)-w'\nabla u\left( \binom{x}{y} \right)\chi\left( w \right)\Bigg) \, f_*^x (N(y,dw))
\end{align*}
where $f^x(\zeta)=\binom{\Phi(x)'\zeta}{\zeta}$ and $u\in C_b^2(\bbr^{d+n})$. If $(X,Y)$ is a Feller process, the usual generator $A$ (as defined in equation \eqref{generator}) coincides with $\widetilde{A}$ on $C_c^2(\bbr^{d+n})$.

Let us remark that the L\'evy measure of the process $(X,Y)$ only lives on the `curve' $\zeta\mapsto \binom{\Phi(x)' \zeta}{\zeta}$. This is natural, since $\Delta X_t=\Phi(X_{t-})' \Delta Y_t$.

\begin{example}
In the situation described in Section 1 where $Y$ is the solution of the L\'evy driven SDE \eqref{levysde} we obtain that the process $(X,Y)$ has the symbol
\begin{align} \label{example}
q\left( \binom{x}{y} , \xi \right) = \Psi(y)'(\Phi(x)'\xi_1 + \xi_2) = \binom{\Phi(x) \Psi(y)}{\Psi(y)} '\xi.
\end{align}
\end{example}

\section{Fine Properties of the Solution}

Coming back to our starting point the coefficients $\Psi$ and $\Phi$ are from now on assumed to be bounded, a condition which is needed in order to obtain the Feller property as well as fine properties. $Y$ denotes the solution of \eqref{levysde} and is the driving term of \eqref{itosde}. 

Even in the L\'evy driven case, it is a difficult question whether the solution is a Feller process (cf. \cite{applebaum} Theorem 6.7.2 and \cite{mydiss} Theorem 2.49). It becomes even more involved for more general driving processes. In the particular case where the driving process is the solution of a L\'evy driven SDE we can plug \eqref{levysde} into \eqref{itosde} in order to obtain:
\begin{align} \begin{split}
dX_t&= \Phi(X_{t-}) \Psi(Y_{t-}) \, dZ_t \\
dY_t&= \Psi(Y_{t-}) \, dZ_t \\
X_0&=x \\
Y_0&=y.
\end{split} \end{align}
This can be written as
\begin{align} \begin{split} \label{multisdeone}
d\binom{X_t}{Y_y} &= \binom{\Phi(X_{t-}) \Psi(Y_{t-})}{\Psi(Y_{t-})} \, dZ_t \\
\binom{X_0}{Y_0} &= \binom{x}{y}.
\end{split} \end{align}
Therefore, the solution is a Feller process by Corollary 3.3 of \cite{sdesymbol}. A simple calculation yields that if $\Phi$ and $\Psi$ are locally Lipschitz, this is as well the case for 
\[
\binom{x}{y} \mapsto \binom{\Phi(x) \Psi(y)}{\Psi(y)}=: \mathfrak{M}(x,y) \in \bbr^{(d+n)\times m}.
\]
Obviously $\mathfrak{M}$ is bounded.

\begin{remark}
This technique of rewriting the SDEs in a $(d+n)$-dimensional way could have been used to obtain the symbol \eqref{example} directly, but only in the case where $Y$ is given as solution of the L\'evy driven SDE. 
\end{remark}

Now we adapt the technical main result of Section 5 of \cite{sdesymbol} to our situation. In order to obtain fine properties of stochastic processes Schilling introduced so called indices (cf. \cite{schilling98}). Since the definition of these indices is rather involved we restrict ourselves here to the upper index $\beta_\infty^x$ which was characterized in \cite{mydiss} Theorem 6.3 in the following way: let $p(x,\xi)$ be a non-trivial (i.e.\ non-constant) symbol of a rich Feller process, then
\[
  \beta_\infty^x = \limsup_{\norm{\eta}\to\infty} \sup_{\norm{w-x}\leq 2/\norm{\eta}} \frac{\log\abs{p(w,\eta)}}{\log\norm{\eta}}.
\]

\begin{proposition} \label{prop:sdeindex}
Let $X$ be the solution process of the SDE \eqref{itosde} and assume $d=n=m$ and
 that $\eta\mapsto\Phi(y)'\eta$ and $\xi\mapsto \Psi(x)'\xi$ are bijective for every $(x,y)\in\bbr^{2m}$. If the driving
L\'evy process has the non-constant symbol $\psi$ and index
$\beta_\infty^\psi$, then the process $(X,Y)$ has, for every
$(x,y)\in\bbr^{2m}$, the upper index
$\beta_\infty^{(x,y)}\equiv\beta_\infty^\psi$.
\end{proposition}

\begin{proof}
We rewrite the SDE \eqref{multisdeone} again in order to obtain a $(2m) \times (2m)$-matrix
\begin{align} \begin{split} \label{multisdetwo}
d\binom{X_t}{Y_t}= \left(\begin{array}{cc} \Phi(X_{t-}) \Psi(Y_{t-}) & 0\\ 
                         0 &  \Psi(Y_{t-}) \end{array}\right)
                  \, d\binom{Z_t}{Z_t}.
\end{split} \end{align}
In particular the driving L\'evy process lives on $\{(x,y)\in\bbr^{2m}: x=y \}$.
The index of the `doubled' L\'evy process $(Z,Z)$ is again $\beta^\psi_\infty$ since
\[
\limsup_{\norm{\binom{\xi_1}{\xi_2}}\to\infty} \frac{\log \abs{\psi(\xi_1+\xi_2)}}{\log\norm{\xi_1+\xi_2}}=\limsup_{\norm{\eta}\to\infty} \frac{\log\abs{\psi(\eta)}}{\log\norm{\eta}}.
\]
As $\xi\mapsto\mathfrak{M}(x,y)'\xi$ is a bijection for every $x,y\in\bbr^m$ we obtain the result by \cite{sdesymbol} Theorem 5.7.
\end{proof}

\begin{remark}
Let us emphasize that in the one-dimensional case the condition bijectivity only means that $\Psi$ and $\Phi$ do never vanish. That is exactly what we encountered in Theorem \ref{thm:mustbelevy}.
\end{remark}

\begin{corollary}
Let the conditions of Proposition \ref{prop:sdeindex} be met. For every $\lambda>\beta^\psi_\infty$ we obtain
\[
\lim_{t\to 0} t^{-1/\lambda} (X-x)^*_t=0 \hspace{10mm} \bbp^x \text{-a.s.}
\]
where $(X-x)^*_t=\sup_{s\leq t} \norm{X_s-x}$.
\end{corollary}

\begin{proof}
Since Proposition \ref{prop:sdeindex} holds the result follows from Corollary 5.11 and the estimate
\begin{align} \label{finepropesti}
\norm{X_t-x}_{1,\bbr^d} \leq \norm{X_t-x}_{1,\bbr^d} + \norm{Y_t-y}_{1,\bbr^n}= \norm{\binom{X_t-x}{Y_t-y}}_{1,\bbr^{d+n}}
\end{align}
where we have used the 1-norm since it makes the calculation most simple.
\end{proof}

The subsequent result follows from \eqref{finepropesti} and \cite{sdesymbol} Corollary 5.10 which relies mainly on a theorem due to Manstavi\v{c}ius \cite{manstavicius}. Let us just recall following definition: 
if $\gamma\in]0,\infty[$ and $g$ is an $\bbr^d$-valued function on the interval $[a,b]$ then
\[
  V^\gamma(g; [a,b]) := \sup_{\pi_n} \sum_{j=1}^n \norm{g(t_j)-g(t_{j-1})}^\gamma
\]
where the supremum is taken over all partitions $\pi_n = (a=t_0 <t_1 < \ldots < t_n =b)$ of $[a,b]$ is called the (strong) $\gamma$-variation of $g$ on $[a,b]$.

\begin{corollary}\label{cor:variation}
Let the conditions of Proposition \ref{prop:sdeindex} be met. Then
\[
  V^\gamma(X; [0,T]) < \infty\quad \bbp^x\text{-a.s.\ for every\ \ } T>0
\]
if $\gamma>\beta^\psi$.
\end{corollary}

We have obtained fine properties of a non-Markov process by using the symbol and the index $\beta^\psi$. In fact one only needs to know the structure of the SDEs and the symbol of the driving L\'evy process in order to obtain these properties. 

Let us just sketch an application in mathematical finance. It is a well known fact that - in the absence of transaction costs - there is an arbitrage (Free Lunch) if for some $\gamma<2$ the $\gamma$-variation is finite and if Stratonovich integrals are considered (cf. e.g. \cite{etheridge} Section 4.1). If one uses a process $X$ (conditioned to stay positive) of the type described above to model asset prices, one should use a driving L\'evy process with index $\beta^\psi = 2$.

\section{A Simulation Study}

For some choices of $Z,\Phi$ and $\Psi$ we have simulated paths of $Z,Y$ and $X$. The simulation of processes given as solutions of SDEs are treated several times and in great detail in the literature (cf. \cite{eulerfeller}, \cite{jacod04} and the references given therein). The crucial point in simulating the solution of such an SDE is always whether it is possible to simulate the increments of the driving term or not. Here we do not want to contribute to this part of the theory but we want to give the reader a flavor on what the process $X$ might look like. We consider the following SDEs:
\begin{align*}
d\binom{Y_t^{(1)}}{Y_t^{(2)}} &= \left(\begin{array}{cc} \sin(Y_t^{(1)}) & 2 Y_t^{(1)} \\ 0 & 1 \end{array}\right)  \, d\binom{Z_t}{t} \\
dX_t&= (\cos(X_t), X_t) \, d\binom{Y_t^{(1)}}{Y_t^{(2)}}.
\end{align*}
In particular $Y_t^{(2)}=t$. Therefore this example shows how to include a driving term $dt$ in the second equation.
For a driving Brownian motion $Z$ traveling at the speed 1000$t$ we get
\begin{center}
\includegraphics[width=45mm]{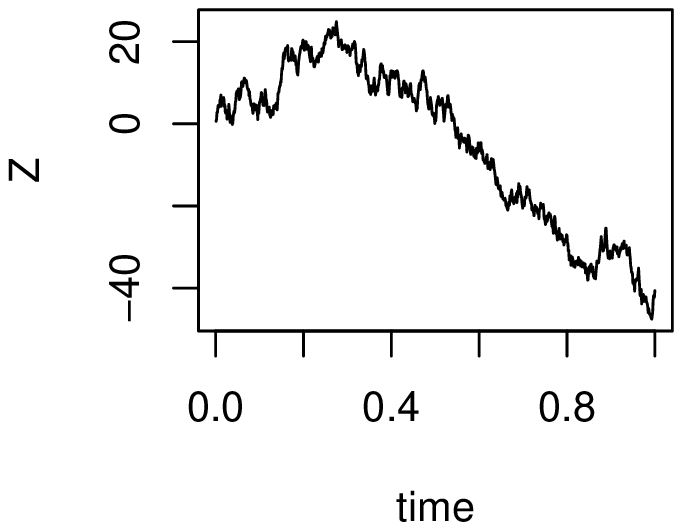}
\includegraphics[width=45mm]{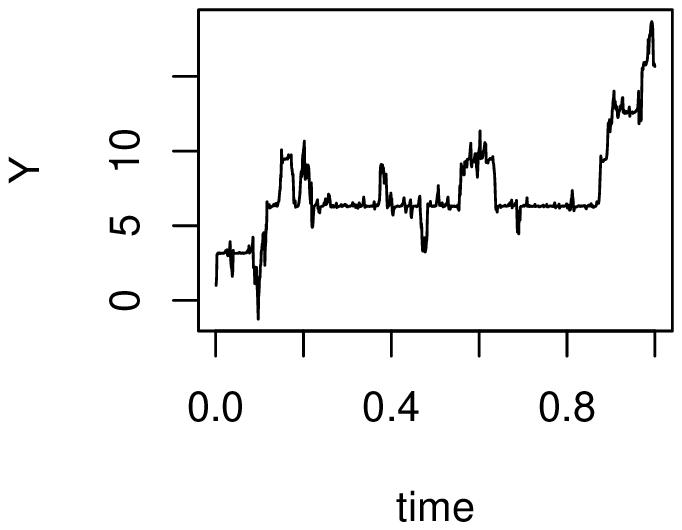}
\includegraphics[width=45mm]{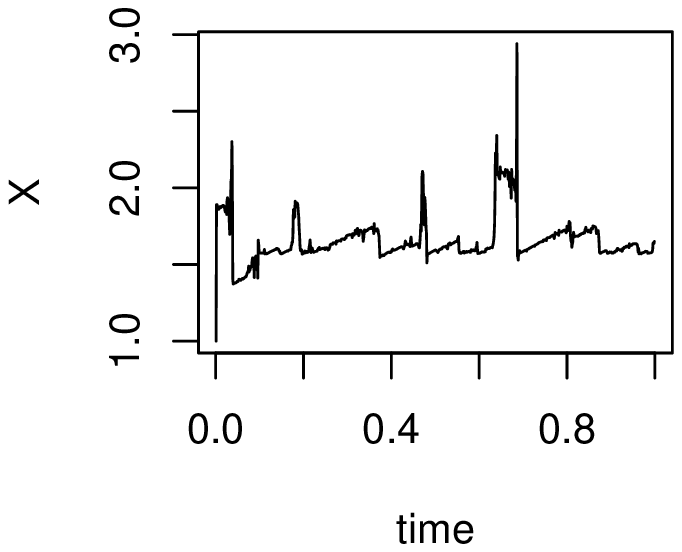}
\end{center}
and for a Cauchy process $Z$ traveling at the speed 1000$t$
\begin{center}
\includegraphics[width=45mm]{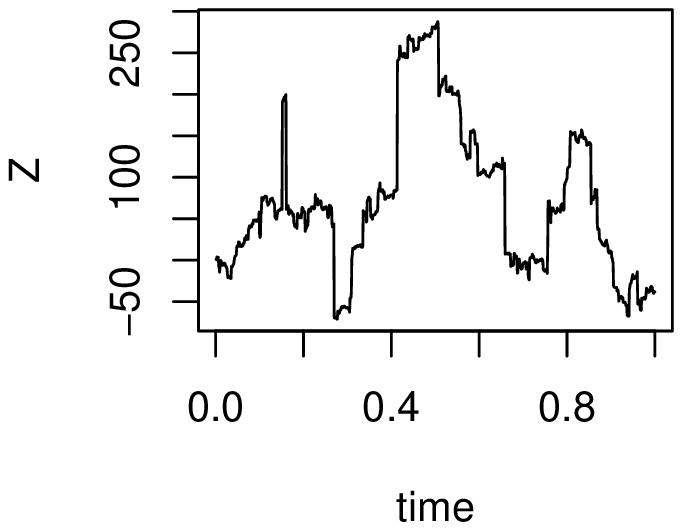}
\includegraphics[width=45mm]{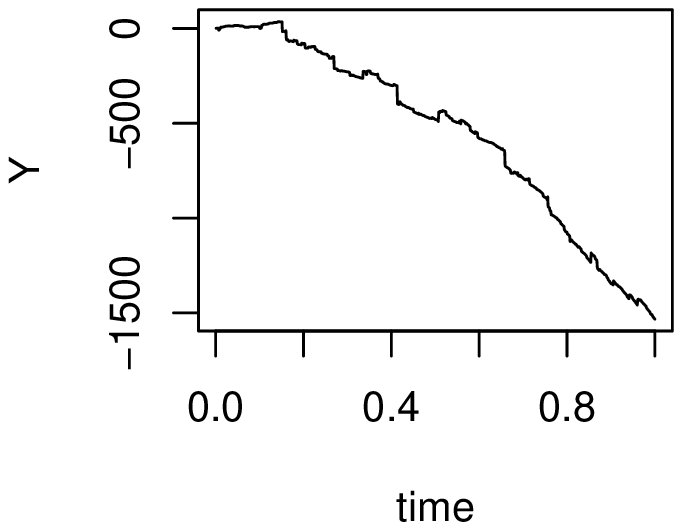}
\includegraphics[width=45mm]{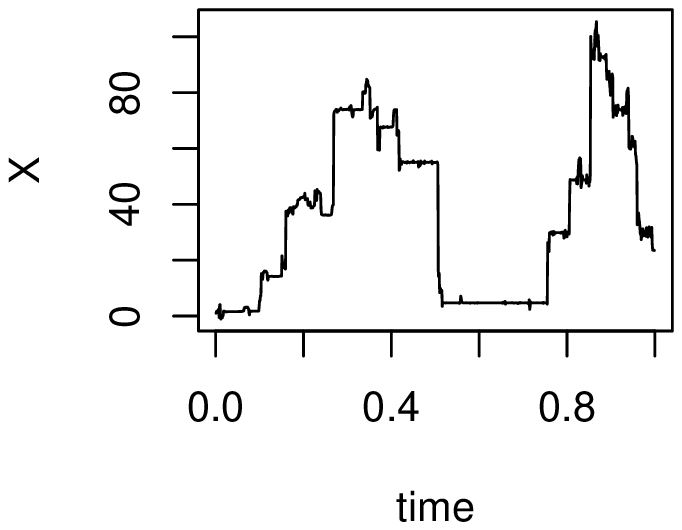}
\end{center}
finally for a Gamma(2) process $Z$ traveling at the speed 100$t$
\begin{center}
\includegraphics[width=45mm]{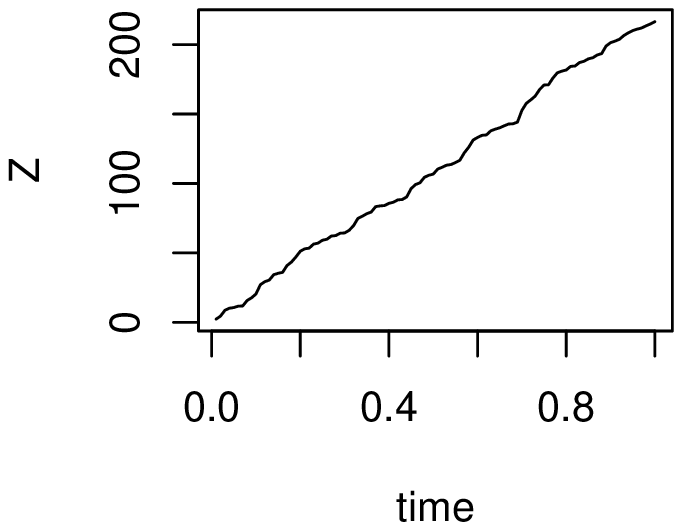}
\includegraphics[width=45mm]{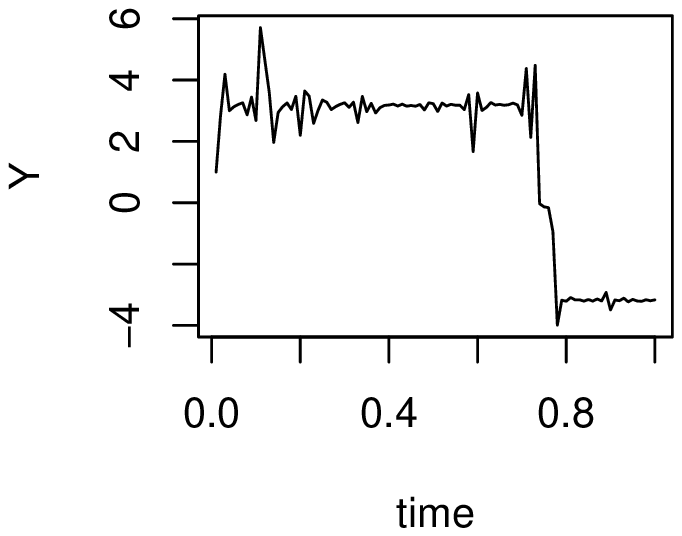}
\includegraphics[width=45mm]{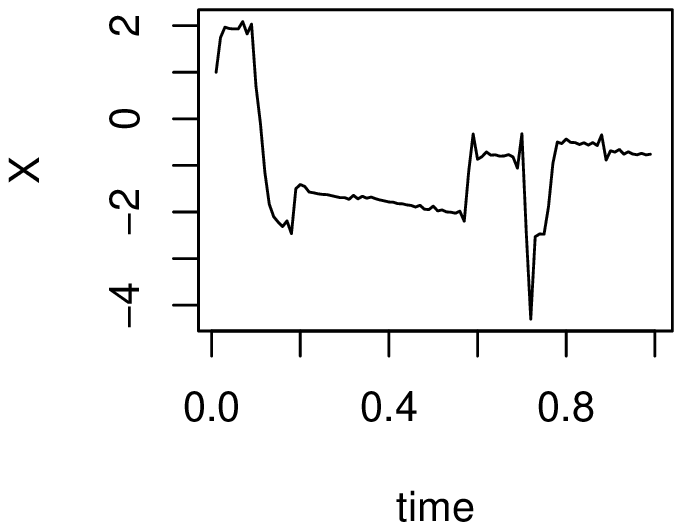}
\end{center}
Let us emphasize that the process $X$, though analytically tractable, shows a different behavior than the It\^o process $Y$. Being precise, the driving process in the examples above is the L\'evy process defined by $(Z_t,t)'$. 

\section{Stochastic Modeling and Complementary Results}

Let us first summerize some of the properties of our process $X$ and emphasize the differences to other L\'evy driven models. First of all the process is automatically a semimartingale. This is interesting in particular for the following two reasons: from the theoretical point-of-view the class of semimartingales is the largest class with respect to which reasonable stochastic integration of predictable processes is possible. This is a classical result which is called Bichteler-Dellacherie theorem (\cite{bichteler} Theorem 7.6). On the other hand, if we want to model asset prices with a locally bounded process, it is again the No Free Lunch property which yields a semimartingale (\cite{del-sch94} Theorem 7.2). In contrast to other models which where introduced recently $X$ can have a part which is of infinite variation on compacts, i.e. a martingale part. This is not the case for the fractional L\'evy processes $M_d$ treated in \cite{ben-lin-sch10}. These processes are semimartingales if and only if they are of finite variation on compacts. In \cite{bas-ped} (Corollary 3.4) this relationship was even shown for general moving average processes $X$ given by $X_t= \int_0^t \phi(t-s) \, dZ_s$ where $Z$ denotes a L\'evy process with unbounded variation, i.e. every such process is an $(\cf_t^Z)_{t\geq 0}$-semimartingale if and only if it is a process of finite variation on compacts. This might lead to problems in modeling asset prices. Compare in this context our remark at the end of Section 3. 

Furthermore the process $X$ admits jumps, this is not the case for classical models which are driven by a Brownian motion or e.g. the well-balanced L\'evy driven Ornstein-Uhlenbeck process which was introduced in \cite{wbLdOU}. Recently there was found strong evidence for the presence of jumps in financial data (cf. e.g. \cite{jacodait}). By choosing the driving L\'evy term in our model appropriately we can control both, the martingale part and the jumps. Let us emphasize that our process is for every starting point an It\^o semimartingale, a class which is often considered in the context of high-frequency financial data (cf. \cite{veraart} and the references given therein). Since the paths allow for splines we might be able to model data from electricity markets which is by now modeled using fractional Brownian motion or ambit fields (cf. \cite{bar-ben-ver} and the references given therein). 

Our process is not Markovian which means that it allows for dependence structures. On the other hand, since it is a `hidden Markov' type of model, it is still analytically tractable and the symbol can be used in order to obtain fine properties, of which some are inherited from the driving term. Last but not least the model offers a lot of flexibility since there is no restriction on the L\'evy process. 


\textbf{Acknowledgments:} I would like to thank an anonymous referee for carefully reading the manuscript and offering useful suggestions which helped to improve the paper.


\end{document}